\documentclass{amsart}

\usepackage{amssymb}

\usepackage{graphicx}
\usepackage{color}
\usepackage[cmtip,all]{xy}
\usepackage{url}
\usepackage{enumitem,kantlipsum}
\usepackage{comment}
\usepackage{float}
\usepackage{tabularray}
\usepackage{array}
\usepackage{enumitem}
\usepackage{svg}
\usepackage{hyperref}
\hypersetup{colorlinks=true,
linkcolor=red,
    filecolor=red,
 urlcolor=red,
 citecolor=red}
\usepackage{setspace}
\newcommand{\HRule}{\rule{\linewidth}{0.2mm}}

\newtheorem{theorem}{Theorem}[section]

\newtheorem{prop}[theorem]{Proposition}

\theoremstyle{definition}
\newtheorem{definition}[theorem]{Definition}

\theoremstyle{remark}
\newtheorem{remark}[theorem]{Remark}

\numberwithin{equation}{section}

\begin{document}

\title[On the Aicardi-Juyumaya bracket for tied links]{On the Aicardi-Juyumaya bracket for tied links}


\author{O'Bryan Cárdenas-Andaur}
\address{Departamento de Álgebra, Facultad de Matemáticas, Instituto de Matemáticas (IMUS), Universidad de Sevilla, Av. Reina Mercedes s/n, 41012 Sevilla, Spain}
\curraddr{}
\email{obryan.cardenas@uv.cl}
\thanks{}


\date{\today}

\dedicatory{}

\begin{abstract}
Aicardi and Juyumaya generalized the Kauffman bracket of classical links to the setting of tied links by introducing the double bracket $\langle\langle\cdot\rangle\rangle$. However, the analogues of Kauffman's states and their relationship to this invariant are not immediately clear. We address this question by defining the \textit{Aicardi-Juyumaya states}, and show that the contribution of each AJ-state to $\langle\langle\cdot\rangle\rangle$ does not depend on the chosen resolution tree. We also present an algorithm to compute the double bracket of a tied link diagram and use it to find pairs of examples of (oriented) tied links sharing the same Homflypt polynomial but different tied Jones polynomial.
\end{abstract}

\maketitle

\section{Introduction}

The Jones polynomial is a link invariant discovered by Vaughan Jones in 1985. This construction is based on a braid group representation over the Temperley-Lieb algebra and a Markov trace over it \cite{Jones1985}. 

On the other hand, Kauffman provided in 1987 a combinatorial definition of the Jones polynomial through the so-called \textit{Kauffman’s method} (see \cite[Theorem 2.6]{Kauffman1987}). For this, he defined what is now known as the Kauffman bracket $\langle \cdot \rangle$, which assigns a polynomial to a link diagram and can be defined as follows:

\begin{align}
\label{suma-estados}
\langle D \rangle=\sum_{s}A^{c-2r}(-A^2-A^{-2})^{k-1},
\end{align}

\noindent where the sum is taken over all Kauffman states $s$ of $D$, $c$ is the number of crossings in $D$, $r$ is the number of 1-smoothings in the state $s$ and $k$ is the number of circles in the smoothed diagram associated to $s$. If we write $w(D)$ for the writhe of a diagram $D$ representing a link $L$, the Jones polynomial can be recovered as
\begin{align*}
    V(L)=(-A)^{-3w(D)}\langle D \rangle.
\end{align*}

Khovanov introduced the first homological invariant, currently known as Khovanov homology, as a categorification of Jones polynomial \cite{Khovanov2000} with stronger properties; for example, it is known that Khovanov homology is an unknot detector \cite{Kronheimer}, while this is an open question for Jones polynomial. In \cite{Viro2004} Viro nicely reinterpreted Khovanov complex in a purely combinatorial way, with generators consisting of enhanced Kauffman states, and differentials defined in terms of adjacency between states. 

Tied links are a generalization of classical links introduced in \cite{Aicardi2016}. 
Every invariant for tied link leads to an invariant for classical links; in fact, tied link invariants are, in general, stronger in the sense that they can distinguish links that are not distinguished by their classical counterpart. 

One of those invariants is the tied Jones polynomial $\mathcal{J}$, introduced in \cite{Aicardi2018}. It is a $2$-variable polynomial that coincides with the Jones polynomial when specialized to classical links. In this sense, the invariant $\mathcal{J}$ not only generalizes the Jones polynomial for tied links but it is also constructed using Kauffman's method, thereby extending \cite[Theorem 2.6]{Kauffman1987}. In a similar way to that of the classical setting, given a tied link $L$ represented by a diagram $D$ one can write
\begin{align*}
    \mathcal{J}(L)=(-A)^{-3w(D)}\langle\langle D \rangle\rangle,
\end{align*}

\noindent where $\langle\langle\cdot\rangle\rangle$ denotes the \textit{Aicardi-Juyumaya bracket}, a generalization of the (classical) Kauffman bracket. However, a similar expression to (\ref{suma-estados}) for $\langle\langle\cdot\rangle\rangle$ is still unknown. 

Given the importance of Khovanov homology and the demonstrated utility of tied link theory in developing stronger invariants, one can expect to build a new version of Khovanov homology for the family of tied links by categorifying the invariant $\mathcal{J}$ (and therefore the AJ-bracket) in a similar way as Viro did. The first step to achieve this would be to understand the differences between the AJ-bracket and the Kauffman bracket. This could seem easy, but this is not the case. For example, there is not a well-defined counterpart of Kauffman states in the case of a tied link diagram, since the leaves (and their number) in a resolution tree of a tied link are not preserved when modifying the order in which crossings are smoothed. In fact, the problem is more subtle: given a tied link diagram, the construction of a (useful) resolution tree is not straightforward, since some smoothings do not decrease the number of crossings and, therefore, a crossing might be smoothed more than once.  

In this paper, we address this problem and introduce \textit{AJ-states} as the counterpart of Kauffman states for the classical setting. Moreover, we prove that given a tied link diagram, the contribution of each AJ-state to the AJ-bracket does not depend on the chosen resolution tree. 

The structure of the paper is as follows: In Section 2 we present tied links and the Aicardi-Juyumaya bracket $\langle\langle\cdot\rangle\rangle$. In Section 3 we introduce AJ-states and prove the main result. We also present an algorithm allowing to compute the AJ-bracket of a given tied link diagram. In Section 4 we discuss some examples showing the strength of the Aicardi-Juyumaya bracket. 

\section{Tied links and Aicardi-Juyumaya bracket}

\subsection{Tied links}

Aicardi and Juyumaya introduced tied links in \cite{Aicardi2016}, as a non-trivial generalization of classical links. These links can be approached from both a topological \cite{Aicardi2016} and an algebraic perspective \cite{Aicardi2018}. In this paper, we will adopt the latter perspective.

\begin{definition}
A tied link is a pair $(L, P)$, where $L$ denotes a classical link and $P$ represents a partition of its components. 
\end{definition}

Denote the set of unoriented tied links by $\widetilde{\mathfrak{L}}$, and write $\mathfrak{L}$ for the set of unoriented classical links. We can see $\mathfrak{L} \subset \widetilde{\mathfrak{L}}$ if we consider a classical link as a tied link where all components belong to the same partition set.

The reason why these objects are called \textit{tied links} is that originally, the components belonging to the same partition were connected by ties, visually represented by dashed lines, as shown in Figure~\ref{examples}.

\begin{figure}[H]
\centering
\includegraphics[height=3.3cm]{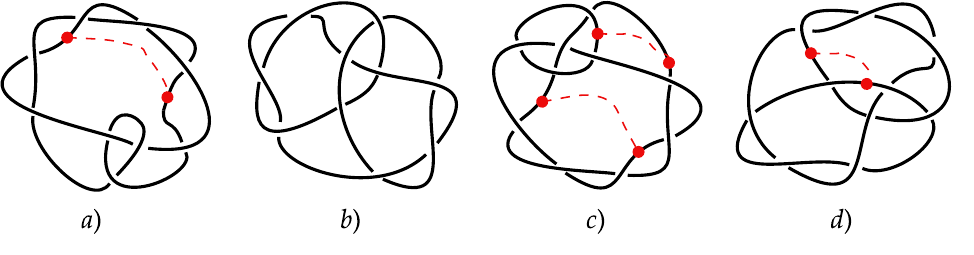}
\caption{Examples of tied link diagrams.}
\label{examples}
\end{figure}

In this paper, we think of a tied link as a colored link whose components are given the same color if and only if they belong to the same set in the partition as shown in Figure~\ref{ties-colores}. Notice that for a tied link $L$ of $\mu(L)$ components, the number of required colors is at most $\mu(L)$. Given a tied link with $m$ partition sets (colors), we can denote the colors as $1,2,\dots,m\in\mathbb{Z}$.

\begin{figure}[H]
\centering
\includegraphics[scale=0.65]{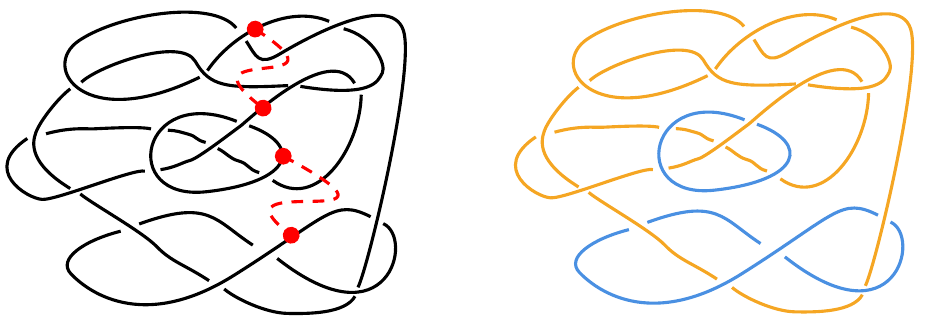}
\caption{Two interpretations of the same tied link diagram.}
\label{ties-colores}
\end{figure}

Taking into account the reinterpretation discussed in the previous paragraph, we can define an equivalence relation based on the definition provided in \cite{Aicardi2016}.

\begin{definition}
Two tied links $(L_1,P_1)$ and $(L_2,P_2)$ are equivalent if:
\begin{enumerate}
\item The links $L_1$ and $L_2$ are ambient isotopic.
\item The ambient isotopy maps $P_1$ to $P_2$.
\end{enumerate}
\end{definition}

A tied link diagram of $(L, P)$ is a projection of $L$ onto $\mathbb{R}^2$, keeping the ties between the components of the projection as they were originally in $(L, P)$.

Two diagrams of a tied link represent the same tied link if it is possible to obtain one from the other through a finite sequence of Reidemeister moves while preserving the colors between the involved components (up to renaming the colors).

\subsection{Aicardi-Juyumaya bracket}

Consider a tied link diagram $D$, and denote by $D\sqcup\bigcirc$ the tied link diagram consisting of the disjoint union of $D$ and the trivial diagram of the unknot whose color does not belong to the set of colors of $D$, and by $D\widetilde{\sqcup}\bigcirc$ the tied link diagram consisting of the disjoint union of $D$ and the trivial diagram of the unknot, whose color coincides with one of the colors in $D$.

The following theorem, proved in \cite{Aicardi2018}, defines a regular isotopy invariant (invariant under Reidemeister moves R2 and R3) for tied links that generalizes the Kauffman bracket for classical links. We will refer to this invariant as the AJ-bracket.

\begin{theorem}\label{invariante}
There exists a unique function
\begin{equation*}
    \langle\langle\cdot\rangle\rangle \ : \widetilde{\mathfrak{L}}\rightarrow\mathbb{Z}[A^{\pm1},c]
\end{equation*}

defined by the following axioms:
\begin{enumerate}
    \item $\langle\langle \bigcirc  \rangle\rangle=1$,
    \item $\langle\langle D \sqcup \bigcirc\rangle\rangle=c\cdot\langle\langle D\rangle\rangle$,
    \item $\langle\langle D \widetilde{\sqcup} \bigcirc\rangle\rangle=-(A^2+A^{-2})\langle\langle D\rangle\rangle$,
    \item $\langle\langle \cdot \rangle\rangle $ is invariant under Reidemeister moves R2 y R3.
    \item $\langle\langle D_{+,\sim} \rangle\rangle =
A\ \langle\langle D_0\rangle\rangle +
 {A^{-1}} \  \langle\langle D_{\infty} \rangle\rangle $,
\item $\langle\langle D_+ \rangle\rangle  +  \langle\langle D_- \rangle\rangle =
\delta\left( \langle\langle D_0 \rangle\rangle +
  \langle\langle D_{\infty} \rangle\rangle \right) $.
\end{enumerate}

Here, $\delta=A+A^{-1}$ and, $D_+, D_-, D_0, D_{\infty}$ and $D_{+,\sim}$ are diagrams differing in the neighborhood of a given crossing, as indicated in Figure~\ref{multidiagram}, where we distinguish whether the two arcs' components belong to the same block (same colors), or to different blocks (different colors).

\begin{figure}[H]
  \centering
  \includegraphics[height=2.5cm]{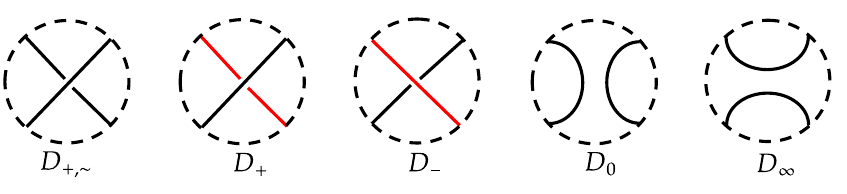}
  \caption{Local diagrams illustrating axioms of Theorem~\ref{invariante}.}
  \label{multidiagram}
\end{figure}
\end{theorem}

Axioms $(1)$, $(3)$, $(4)$ and $(5)$ in Theorem~\ref{invariante} are inherited from that of the classical setting.  

Observe that given a tied link diagram $D$, it is not possible to distinguish whether a crossing whose arcs do not share the same color corresponds to a diagram of type $D_+$ or $D_-$. We fix this ambiguity in the following way: assume the colors of $D$ are labelled as $1, 2, \ldots, m$, and consider a crossing whose upper and lower arcs are labelled by $i$ and $j$, respectively. We declare the diagram to be of type $D_+$ for the distinguished crossing if $i<j$; otherwise, the diagram is of type $D_-$. This will become crucial when constructing resolution trees in the next section.

If we have a crossing of type $D_{+,\sim}$, we will denote by $\bar{0}$ and $\bar{1}$, respectively, the smoothings to obtain the diagrams $D_0$ and $D_{\infty}$ when applying axiom $(5)$. On the other hand, if we have a crossing of type $D_{+}$, we will denote by 0, 1, and 2, respectively, the smoothings to obtain the diagrams $D_0$, $D_{\infty}$, and $D_-$ when applying axiom $(6)$. Figure~\ref{smoothing} illustrates this notation:

\begin{figure}[H]
    \centering
    \includegraphics[scale=0.75]{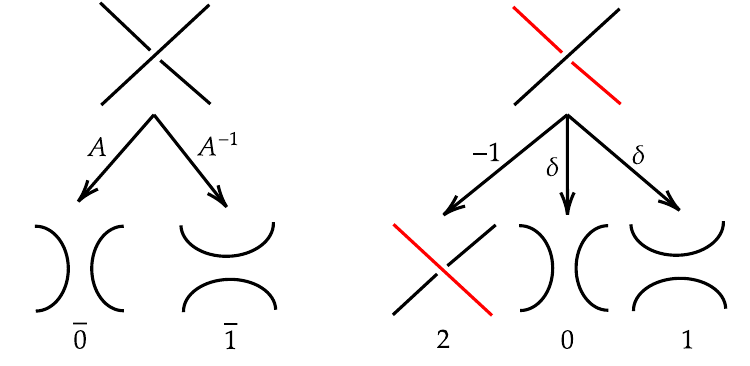}
    \caption{Types of smoothings, asumming color black and red represent colors $i$ and $j$, with $i<j$.}
    \label{smoothing}
\end{figure}

\begin{remark}\label{low-complexity}
Observe that when we apply a smoothing of type $\Bar{0}$, $\Bar{1}$, or $2$, the colors of the involved arcs are preserved. However, when smoothings are of type $0$ or $1$, the colors of the arcs of the distinguished crossings merge into one, that is, the two associated blocks of the partition are joined. We stablish the arcs in $D_0$ and $D_{\infty}$ to inherit the color of lower index, as shown in Figure~\ref{smoothing}. 
\end{remark}

Recall that to compute the Kauffman bracket of a classical link diagram $D$, we can use a \textit{resolution tree}. This involves sequentially taking the crossings of $D$ and applying axiom $(5)$ in Theorem~\ref{invariante}. This process yields two diagrams with one less crossing each time. These diagrams are connected to $D$ with edges (branches) labeled with the coefficient associated with the type of smoothing. By iterating this process, we end up with a (resolution) tree whose leaves are disjoint unions of circles without crossings (the so called Kauffman states). If $D$ had $c$ crossings, any resolution tree of $D$ contains $2^c$ leaves. In order to get the Kauffman bracket of each Kauffman state we can use axioms $(1)$ and $(3)$. Finally, to get the Kauffman bracket polynomial of $D$ we multiply the value of the bracket of each Kauffman state by the product of the labels of the edges connecting it to the root $D$, and take the sum over all leaves:

\begin{align*}
\langle D \rangle=\sum_{s}A^{c-2r}(-A^2-A^{-2})^{k-1},
\end{align*}

Here, $(-A^2-A^{-2})^{k-1}$ is the value of the bracket for each Kauffman state $s$, where $k$ is the number of circles in the smoothed state associated to $s$, and $A^{c-2r}$ is the product of the labels of the edges connecting $s$ to the root, with $c$ the number of crossings of $D$ and $r$ the number of $\overline{1}$-smoothings in the state $s$.

As opposed to the previous case, when performing the resolution tree of a tied link for the AJ-bracket, we have two axioms that we can use depending on the type of crossing. If we use axiom $(5)$, the number of crossings decreases just as in the classical case. However, when using axiom $(6)$, one of the resulting diagrams has the same number of crossings, potentially leading to a nonterminating construction of the resolution tree. In the next section, we will introduce the concept of the complexity of a tied link diagram, which will, in turn, allow us to define the set of AJ-states and ensure that the construction of a resolution tree for a given diagram indeed terminates. 

\section{AJ-states and computation of $\langle\langle\cdot\rangle\rangle$}

In this section we introduce AJ-states and prove that, given a tied link diagram $D$, the contribution of each AJ-state to $\langle\langle D\rangle\rangle$ does not depend on the chosen resolution tree. 

\subsection{Complexity of tied link diagrams}

\begin{definition}
Let $D$ be a tied link diagram in which the colors of the components have been indexed by $\lbrace 1,2,3,\dots, m \rbrace$. We define the complexity of $D$ as the pair $(C_T, C_I)$, where $C_T$ corresponds to the total number of crossings of $D$ and $C_I$ corresponds to the number of \textit{illegal crossings}, which can be of two types:

\begin{enumerate}
    \item Type 1: Crossings whose arcs share the same color.
    \item Type 2: Crossings whose arcs are of different colors and satisfy that the color of the upper arc has lower index than that of the lower arc. 
\end{enumerate}  
\end{definition}

Diagram $D$ shown in Figure~\ref{resolution-tree} has $4$ illegal crossings: two of them are of type $2$ (color $2$ (blue) is below color $1$ (black)) and two of them of type $1$ (involving black arcs). Therefore, the complexity of $D$ is $(6,4)$.

Furthermore, since $(C_T, C_I)\in\mathbb{N}^2$, we can endow the set of complexities of tied link diagrams with the lexicographic order.

\begin{definition}\label{estado final}
Given a tied link diagram $D$, the Aicardi-Juyumaya states (AJ-states) of $D$ are all diagrams of complexity $(C_T,0)$ that can be obtained from $D$ by a finite sequence of smoothings of illegal crossings using axioms $(5)$ and $(6)$ in Theorem~\ref{invariante}.
\end{definition}

Observe that, since some smoothings imply a change of colors in the components of the diagram (see Remark~\ref{low-complexity}), it might happen that a crossing that is illegal in $D$ become legal once we smooth other crossing; conversely, a legal crossing in $D$ could become illegal at some step when constructing resolution tree.  

From the above definition, it follows that the AJ-states for tied link diagrams contain and generalize the Kauffman states, adding to these configurations of possibly overlapping circles of distinct colors, such as those named $D_5$ to $D_8$ in Figure~\ref{resolution-tree}.

\begin{figure}[H]
    \centering
    \includegraphics[scale=0.45]{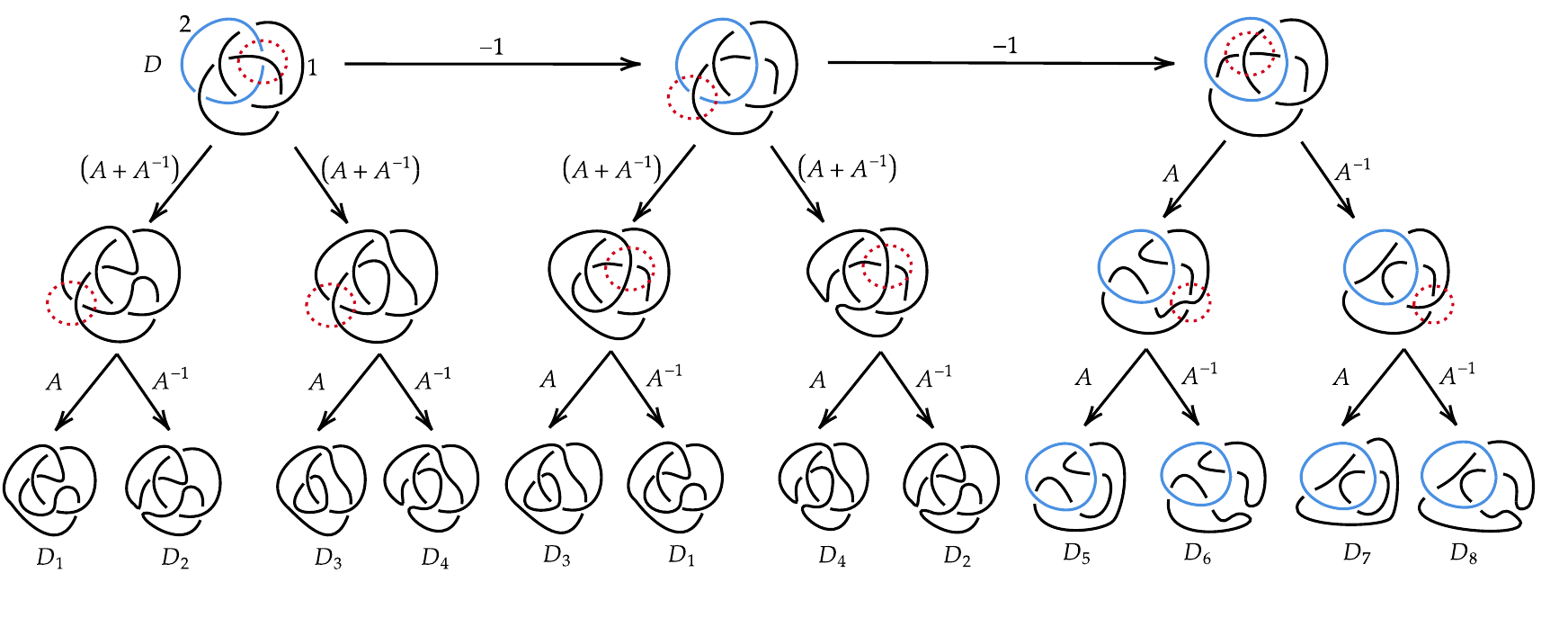}
    \caption{Diagrams $D_5$ to $D_8$ are AJ-states but $D_1$ to $D_4$ are not.}
    \label{resolution-tree}
\end{figure}

The previous figure is an example of an incomplete resolution tree. To complete the resolution tree, one just needs to apply the classical resolution procedure to the diagrams $D_1$ to $D_4$.

\begin{remark}\label{low-complexity2}
From Remark~\ref{low-complexity}, we know that we reduce the complexity of tied link diagrams when using axioms $(5)$ and $(6)$ until reaching the AJ-states. Indeed, we have that once an illegal crossing $c$ in a diagram of complexity $(n,k)$ is smoothed, if it is of type 1, then two diagrams of complexity $(n-1,k-1)$ are obtained. On the other hand, if $c$ is of type 2, a diagram of complexity $(n,k-1)$ and two diagrams of complexity $(n-1,k')$ with $k'\in\mathbb{N}$ are obtained. In both cases, one obtains diagrams with lower complexity. Observe that the change of colors of some components might cause $k'$ to be greater than $k$.
\end{remark}

The importance of AJ-states lies in the fact that they are the leaves of a resolution tree for $D$, and we can calculate the invariant $\langle\langle D\rangle\rangle$ by computing the value associated to each leaf, which is easily done by axioms $(1)$ to $(4)$ in Theorem~\ref{invariante}.

\subsection{Contribution of an $AJ$-state to $\langle\langle D\rangle\rangle$}

Let $D$ be a tied link diagram, and $R$ be a resolution tree with root $D$. Given a vertex of $R$, there exists a unique path from the root to that vertex. This path consists of smoothing the crossings $c_1, \dots, c_k$ of $D$ through smoothings of type $a_1, \dots, a_k$, with $a_i\in\lbrace 0,1,2,\Bar{0},\Bar{1}\rbrace$, respectively. We will denote this vertex as $D_{c_1(a_1),\dots, c_k(a_k)}$ and its associated diagram as $s(D_{c_1(a_1),\dots, c_k(a_k)})$. Notice that distinct vertices could have the same associated diagram. We will idenfity $s(D)$ and $D$. Also, notice that if $v$ is a leaf of $R$, then $s(v)$ is an AJ-state.

Let $S$ be the set of AJ-states of $D$. For each $s\in S$ we define its associated polynomial as $f_R(s)=\displaystyle\sum_{v \in R: s(v)=s} b(v)$, where $b(v)$ is the product of labels of the branches connecting the vertex $v$ to the root of $R$. If $s\neq s(v)$ for every $v\in R$, we set $f_R(s)=0$.

Given a resolution tree $R$, we define $T_R(D)=\displaystyle\sum_{s\in S} f_R(s)\cdot s$ \, as the formal sum of the AJ-states multiplied by their associated polynomials. Furthermore, if $v$ is a vertex of $R$, we will call $R_v$ the subtree with root $v$, and we define $T_R(v)=T_{R_v}(s(v))$.

\begin{theorem}\label{independencia}
Given a tied link diagram $D$, $T_R(D)$ does not depend on the resolution tree $R$, that is, it does not depend on the specific sequence of smoothings of the crossings.
\end{theorem}

\begin{proof}
To prove the theorem, we will consider two resolution trees $R_1$ and $R_2$ with root $D$, and proceed by induction on the complexity of $D$.

The base case is to consider a diagram of complexity $(0,0)$; in this case, there are no illegal crossings, so $D$ is an AJ-state and the statement holds because $T_{R_1}(D)=D=T_{R_2}(D)$. Suppose that we have a diagram of complexity $(n,k)$ and that for any diagram of lower complexity the statement is true; we have the following possible cases:

If $k=0$, then, by definition, we are dealing with an AJ-state, and $R_1=R_2$, so $T_{R_1}(D)=D=T_{R_2}(D)$.

If $k=1$, we have two cases: the unique illegal crossing $c$ is of type 1 or of type 2. If it is of type 1, then:

\begin{align*}
    T_{R_1}(D)&=AT_{R_1}(D_{c(\Bar{0})})+A^{-1}T_{R_1}(D_{c(\Bar{1})})\\
    T_{R_2}(D)&=AT_{R_2}(D_{c(\Bar{0})})+A^{-1}T_{R_2}(D_{c(\Bar{1})})
\end{align*}

Then, by Remark~\ref{low-complexity2}, $s(D_{c(\Bar{0})})$ and $s(D_{c(\Bar{1})})$ have a lower complexity than $D$, so $T_{R_1}(D_{c(\Bar{0})}))=T_{R_2}(D_{c(\Bar{0})})$ and $T_{R_1}(D_{c(\Bar{1})}))=T_{R_2}(D_{c(\Bar{1})})$ by the induction hypothesis, and therefore $T_{R_1}(D)=T_{R_2}(D)$.

If the illegal crossing is of type 2, the proof is analogous.
    
Now, suppose that $k>1$. Notice that if both resolution trees have the same root $D$ and start by smoothing the same illegal crossing, then we have a case analogous to $k=1$.
If resolution trees start by smoothing distinct illegal crossings, say $l$ and $r$, then we need to consider the possible combinations of types of illegal crossings.

\begin{enumerate}
            \item Crossings $r$ and $l$ are of type 1 and are of different colors: If $R_1$ starts by smoothing the crossing $l$:

\begin{align*}
    T_{R_1}(D) &=AT_{R_1}(D_{l(\Bar{0})})+A^{-1}T_{R_1}(D_{l(\Bar{1})})
\end{align*}

 By the induction hypothesis, to compute $T_{R_1}(D_{l(\Bar{0})})$ and $T_{R_1}(D_{l(\Bar{1})})$ we can replace the subtrees $T_{R_1}(D_{l(\Bar{0})})$ and $T_{R_1}(D_{l(\Bar{1})})$ by other resolution trees of the corresponding vertices. Since $r$ is still an illegal crossing of $D_{l(\Bar{0})}$ and $D_{l(\Bar{1})}$, we can assume that $T_{R_1}(D_{l(\Bar{0})})$ and $T_{R_1}(D_{l(\Bar{1})})$ both start by smoothing $r$. Thus, we obtain

\begin{align*}
    T_{R_1}(D) &=A^2 T_{R_1}(D_{l(\Bar{0}),r(\Bar{0})})+ T_{R_1}(D_{l(\Bar{0}),r(\Bar{1})})+ T_{R_1}(D_{l(\Bar{1}),r(\Bar{0})})+ A^{-2}T_{R_1}(D_{l(\Bar{1}),r(\Bar{1})})
\end{align*}

            Similarly, if $R_2$ starts by smoothing the crossing $r$, we can assume that $T_{R_2}(D_{r(\Bar{0})})$ and $T_{R_2}(D_{r(\Bar{1})})$ both start by smoothing $l$ and we obtain the following.

            \begin{align*}
                T_{R_2}(D)&=AT_{R_2}(D_{r(\Bar{0})})+A^{-1}T_{R_2}(D_{r(\Bar{1})})\\
                &=A^2 T_{R_2}(D_{r(\Bar{0}),l(\Bar{0})})+ T_{R_2}(D_{r(\Bar{0}),l(\Bar{1})})+ T_{R_2}(D_{r(\Bar{1}),l(\Bar{0})})+ A^{-2}T_{R_2}(D_{r(\Bar{1}),l(\Bar{1})})
            \end{align*}

            Now notice that $s(D_{r(i),l(j)})=s(D_{l(j),r(i)})$ for every $i, j \in \lbrace\Bar{0},\Bar{1}\rbrace$, and all these diagrams have smaller complexity than $D$. Hence, by induction hypothesis, $T_{R_1}(D)=T_{R_2}(D)$.

            \begin{figure}[H]
                \centering
                \includegraphics[scale=0.9]{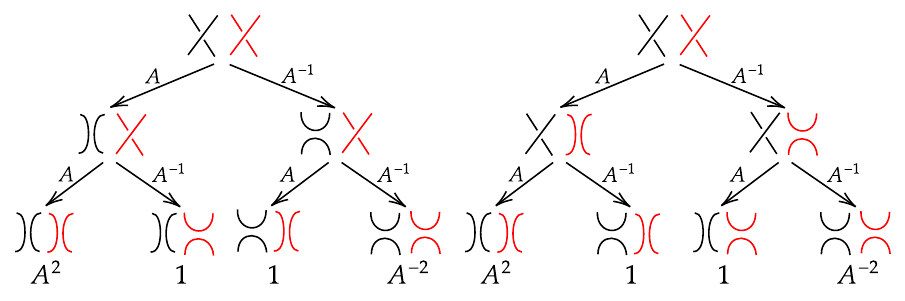}
                \caption{Resolution trees illustrating case (1).}
                \end{figure}

If we consider both illegal crossings of type 1 and they share the same color the procedure is analogous.\\

            \item Crossing $r$ is of type 2 and crossing $l$ is of type 1:\\
            \begin{enumerate}
                \item One arc of $r$ is of the same color as the arcs in $l$: If $R_1$ starts by smoothing the crossing $l$ of type 1:

\begin{align*}
    T_{R_1}(D) &=AT_{R_1}(D_{l(\Bar{0})})+A^{-1}T_{R_1}(D_{l(\Bar{1})})\\
    &=- A  T_{R_1}(D_{l(\Bar{0}),r(2)})+ \delta A  T_{R_1}(D_{l(\Bar{0}),r(0)})+\delta A  T_{R_1}(D_{l(\Bar{0}),r(1)})\\
    & \ \ - A^{-1}  T_{R_1}(D_{l(\Bar{1}),r(2)})+ \delta A^{-1}  T_{R_1}(D_{l(\Bar{1}),r(0)}) +\delta A^{-1}  T_{R_1}(D_{l(\Bar{1}),r(1)})
\end{align*}

We have used the induction hypothesis in the second equality, as $D_{l(\Bar{0})}$ and $D_{l(\Bar{1})}$ have lower complexity than $D$, we can assume that their correlated subtrees start by smoothing $r$, which is still an illegal crossing in these diagrams. Likewise, if $R_2$ starts by smoothing the crossing $r$ of type 2, we obtain the following.

            \begin{align*}
                T_{R_2}(D)&=- T_{R_2}(D_{r(2)}) + \delta T_{R_2}(D_{r(0)}) + \delta T_{R_2}(D_{r(1)})  \\
                &= - A T_{R_2}(D_{r(2),l(\Bar{0})}) - A^{-1} T_{R_2}(D_{r(2),l(\Bar{1})}) \\
                & \ \ +\delta A T_{R_2}(D_{r(0),l(\Bar{0})}) + \delta A^{-1} T_{R_2}(D_{r(0),l(\Bar{1})}) \\ 
                & \ \ + \delta A T_{R_2}(D_{r(1),l(\Bar{0})}) + \delta A^{-1} T_{R_2}(D_{r(1),l(\Bar{1})})
            \end{align*}

            Again, we used the induction hypothesis in the second equality, as in $D_{r(2)}$, $D_{r(0)}$ and $D_{r(1)}$ $l$ is an illegal crossing of type 1. We then notice that $s(D_{r(i),l(j)})=s(D_{l(j),r(i)})$ for every $i \in \lbrace 0,1,2\rbrace$ and $j \in \lbrace\Bar{0},\Bar{1}\rbrace$, so $T_{R_1}(D_{r(i),l(j)})=T_{R_2}(D_{l(j),r(i)})$ by induction hypothesis, therefore $T_{R_1}(D)=T_{R_2}(D)$.

                \begin{figure}[H]
                \centering
                \includegraphics[scale=0.8]{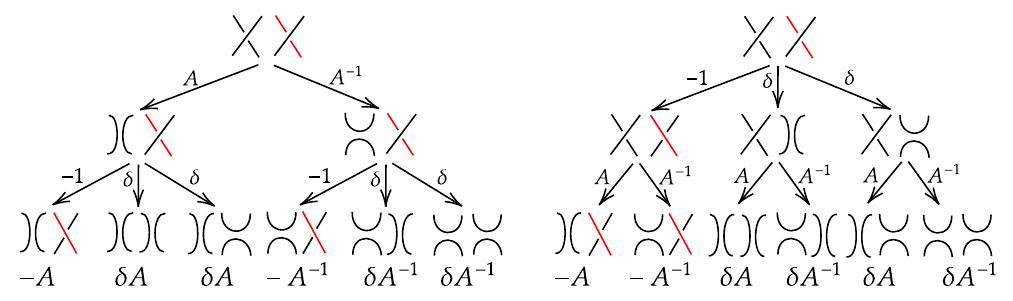}
                \caption{Resolution trees illustrating subcase (2a) in the case when the common color is the one of the upper arc of $r$. In the other case the tree is analogous.}
                \end{figure}

                \item If the arcs in $l$ and $r$ do not share any color, then the proof is analogous to that of case (2a).\\

            \end{enumerate}

            \item Crossings $r$ and $l$ are of type 2: If $r$ and $l$ do not share any color, then we can start by smoothing either crossing, and the induction hypothesis can be applied thanks to Remark~\ref{low-complexity2}. Then, we will consider the cases when the arcs of $r$ and $l$ involve two or three colors. In all cases, we will choose the crossings $l$ and $r$ so that they match the subcases of Figure~\ref{parejas}. 
            
            \begin{figure}[H]
                \centering
                \includegraphics[scale=0.75]{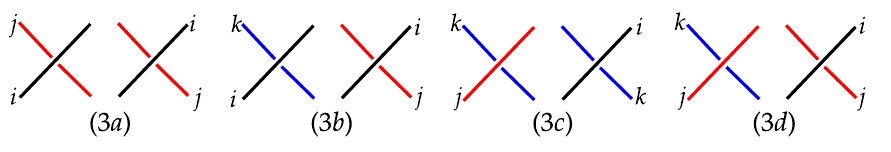}
                \caption{Pairs of colors of crossings $l$ (left) and $r$ (right) illustrating the proof of case (3). We assume $i<j<k$.}
                \label{parejas}
                \end{figure}
            
            \begin{enumerate}

                    \item Both illegal crossing have arcs with the same colors, then:

If $R_1$ starts by smoothing the crossing $l$:

\begin{align*}
    T_{R_1}(D) &=- T_{R_1}(D_{l(2)}) + \delta T_{R_1}(D_{l(0 )}) + \delta T_{R_1}(D_{l(1)}) \\
            &= T_{R_1}(D_{l(2),r(2)}) - \delta T_{R_1}(D_{l(2),r(0)}) - \delta T_{R_1}(D_{l(2),r(1)})  \\ 
            & \ \ +\delta A T_{R_1}(D_{l(0),r(\Bar{0})}) + \delta A^{-1} T_{R_1}(D_{l(0),r(\Bar{1})}) \\
            & \ \ + \delta A T_{R_1}(D_{l(1),r(\Bar{0})}) + \delta A^{-1} T_{R_1}(D_{l(1),r(\Bar{1})})\\
            &= T_{R_1}(D_{l(2),r(2)}) - \delta A T_{R_1}(D_{l(2),r(0),l(\Bar{0})}) - \delta A^{-1} T_{R_1}(D_{l(2),r(0),l(\Bar{1})}) \\
            & \ \ - \delta A T_{R_1}(D_{l(2),r(1),l(\Bar{0})}) - \delta A^{-1} T_{R_1}(D_{l(2),r(1),l(\Bar{1})})  \\ 
            & \ \ +\delta A T_{R_1}(D_{l(0),r(\Bar{0})}) + \delta A^{-1} T_{R_1}(D_{l(0),r(\Bar{1})}) \\
            & \ \ + \delta A T_{R_1}(D_{l(1),r(\Bar{0})}) + \delta A^{-1} T_{R_1}(D_{l(1),r(\Bar{1})}) \\
            &= T_{R_1}(D_{l(2),r(2)})  - \delta A^{-1} T_{R_1}(D_{l(2),r(0),l(\Bar{1})}) - \delta A T_{R_1}(D_{l(2),r(1),l(\Bar{0})})   \\ 
            & \ \ +\delta A T_{R_1}(D_{l(0),r(\Bar{0})}) + \delta A^{-1} T_{R_1}(D_{l(1),r(\Bar{1})})
\end{align*}

We have used the induction hypothesis in the second equality, as $D_{l(0)}$, $D_{l(1)}$ and $D_{l(2)}$ have smaller complexity than $D$, we can choose to smooth the crossing $r$, which becomes an illegal crossing of type 1 in $D_{l(0)}$ and  $D_{l(1)}$. Observe that $l$ becomes and illegal crossing of type 1 in $D_{l(2),r(0)}$ and $D_{l(2),r(1)}$, and therefore we smooth it to obtain the third equality. Finally, to obtain the last expression, we note that $s(D_{l(2),r(0),l(\Bar{0})})=s(D_{l(1),r(\Bar{0})})$ and that $s(D_{l(2),r(1),l(\Bar{1})})=s(D_{l(0),r(\Bar{1})})$.

Similarly, if $R_2$ starts by smoothing the crossing $r$, we obtain the following:

\begin{align*}
    T_{R_2}(D) &=- T_{R_2}(D_{r(2)})+ \delta T_{R_2}(D_{r(0)}) + \delta T_{R_2}(D_{r(1)}) \\
            &= T_{R_2}(D_{r(2),l(2)}) - \delta T_{R_2}(D_{r(2),l(0)}) - \delta T_{R_2}(D_{r(2),l(1)}) \\
            & \ \ +\delta A T_{R_2}(D_{r(0),l(\Bar{0})}) + \delta A^{-1} T_{R_2}(D_{r(0),l(\Bar{1})}) \\
            & \ \ + \delta A T_{R_2}(D_{r(1),l(\Bar{0})}) + \delta A^{-1} T_{R_2}(D_{r(1),l(\Bar{1})}) \\
            &= T_{R_2}(D_{r(2),l(2)}) - \delta A T_{R_2}(D_{r(2),l(0),r(\Bar{0})}) - \delta A^{-1} T_{R_2}(D_{r(2),l(0),r(\Bar{1})})\\
            & \ \ - \delta A T_{R_2}(D_{r(2),l(1),r(\Bar{0})}) - \delta A^{-1} T_{R_2}(D_{r(2),l(1),r(\Bar{1})}) \\
            & \ \ +\delta A T_{R_2}(D_{r(0),l(\Bar{0})}) + \delta A^{-1} T_{R_2}(D_{r(0),l(\Bar{1})}) \\
            & \ \ + \delta A T_{R_2}(D_{r(1),l(\Bar{0})}) + \delta A^{-1} T_{R_2}(D_{r(1),l(\Bar{1})}) \\
            &= T_{R_2}(D_{r(2),l(2)})  - \delta A^{-1} T_{R_2}(D_{r(2),l(0),r(\Bar{1})})- \delta A T_{R_2}(D_{r(2),l(1),r(\Bar{0})}) \\
            & \ \ +\delta A T_{R_2}(D_{r(0),l(\Bar{0})}) + \delta A^{-1} T_{R_2}(D_{r(1),l(\Bar{1})})
\end{align*}

 We then notice that $s(D_{l(2),r(2)})=s(D_{r(2),l(2)})$, $s(D_{l(2),r(0),l(\Bar{1})})=s(D_{r(2),l(0),r(\Bar{1})})$,  $s(D_{l(2),r(1),l(\Bar{0})})=s(D_{r(2),l(1),r(\Bar{0})})$,  $s(D_{l(0),r(\Bar{0})})=s(D_{r(0),l(\Bar{0})})$ and $s(D_{l(1),r(\Bar{1})})=s(D_{r(1),l(\Bar{1})})$ , therefore $T_{R_1}(D)=T_{R_2}(D)$.

                    \begin{figure}[H]
               \centering
                \includegraphics[scale=0.8]{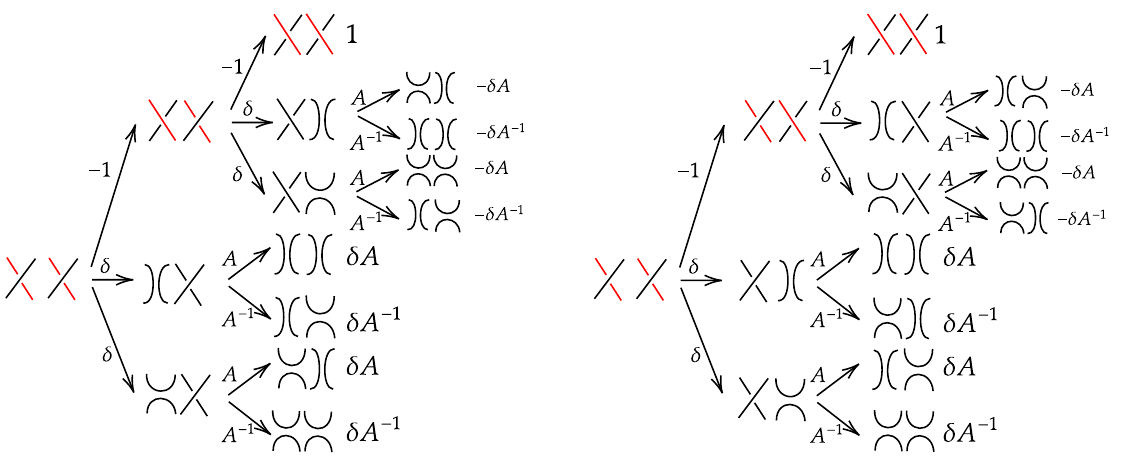}
               \caption{Resolution trees illustrating subcase (3a).}
                \end{figure}
                
                    \item The upper arcs of $r$ and $l$ share the same color, while the lower arcs have different colors. If $R_1$ starts by smoothing the crossing $l$:

\begin{align*}
    T_{R_1}(D) &=- T_{R_1}(D_{l(2)}) + \delta T_{R_1}(D_{l(0)}) + \delta T_{R_1}(D_{l(1)}) \\
            &=  T_{R_1}(D_{l(2),r(2)}) - \delta T_{R_1}(D_{l(2),r(0)}) - \delta T_{R_1}(D_{l(2),r(1)})   \\
            & \ \ - \delta T_{R_1}(D_{l(0),r(2)}) + \delta^2 T_{R_1}(D_{l(0),r(0)}) + \delta^2 T_{R_1}(D_{l(0),r(1)})   \\
            & \ \ - \delta T_{R_1}(D_{l(1),r(2)}) + \delta^2 T_{R_1}(D_{l(1),r(0)}) + \delta^2 T_{R_1}(D_{l(1),r(1)}) 
\end{align*}

Observe that the arcs in crossing $r$ preserve their colors when smoothing $l$ in all possible ways. Similarly, if $R_2$ starts by smoothing the crossing $r$, we obtain the following:

            \begin{align*}
                T_{R_2}(D) &= - T_{R_2}(D_{r(2)}) + \delta T_{R_2}(D_{r(0)}) + \delta T_{R_2}(D_{r(1)})  \\
&=  T_{R_2}(D_{r(2),l(2)}) - \delta T_{R_2}(D_{r(2),l(0)}) - \delta T_{R_2}(D_{r(2),l(1)}) \\
& \ \ - \delta T_{R_2}(D_{r(0),l(2)}) +\delta^2 T_{R_2}(D_{r(0),l(0)}) + \delta^2 T_{R_2}(D_{r(0),l(1)})   \\
& \ \ - \delta T_{R_2}(D_{r(1),l(2)}) + \delta^2 T_{R_2}(D_{r(1),l(0)}) + \delta^2 T_{R_2}(D_{r(1),l(1)}) 
            \end{align*}

            We then see that $s(D_{r(i),l(j)})=s(D_{l(j),r(i)})$ for every $i,j \in \lbrace 0,1,2\rbrace$, and the induction hypothesis completes the proof. 

                    \begin{figure}[H]
                \centering
                \includegraphics[scale=0.8]{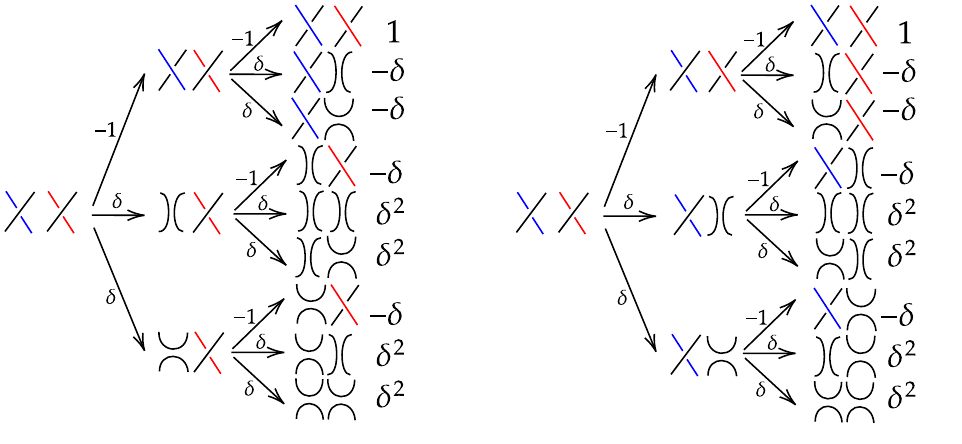}
                \caption{Resolution trees illustrating subcase (3b).}
                \end{figure}

                    \item The lower arcs of $r$ and $l$ share the same color $k$ while the upper arcs of $l$ and $r$ are colored by $j$ and $i$, with $i<j<k$:

                    If $R_1$ starts by smoothing the crossing $l$:

\begin{align*}
    T_{R_1}(D) &=- T_{R_1}(D_{l(2)}) + \delta T_{R_1}(D_{l(0)}) + \delta T_{R_1}(D_{l(1)}) \\
            &=  T_{R_1}(D_{l(2),r(2)}) - \delta T_{R_1}(D_{l(2),r(0)}) - \delta T_{R_1}(D_{l(2),r(1)})  \\
            & \ \ - \delta T_{R_1}(D_{l(0),r(2)}) +\delta^2 T_{R_1}(D_{l(0),r(0)}) + \delta^2 T_{R_1}(D_{l(0),r(1)})  \\ 
            & \ \  - \delta T_{R_1}(D_{l(1),r(2)}) + \delta^2 T_{R_1}(D_{l(1),r(0)})  + \delta^2 T_{R_1}(D_{l(1),r(1)}) 
\end{align*}


Observe that in $D_{l(2),r(0)}$ and $D_{l(2),r(1)}$ color $k$ (blue) becomes $i$ (black), and thereforo crossing $l$ becomes illegal of type $2$, and we can smooth it to obtain:

\begin{align*}
    T_{R_1}(D) &=T_{R_1}(D_{l(2),r(2)})+ \delta T_{R_1}(D_{l(2),r(0),l(2)}) - \delta^2 T_{R_1}(D_{l(2),r(0),l(0)}) - \delta^2 T_{R_1}(D_{l(2),r(0),l(1)}) \\
            & \ \ + \delta T_{R_1}(D_{l(2),r(1),l(2)}) - \delta^2 T_{R_1}(D_{l(2),r(1),l(0)}) - \delta^2 T_{R_1}(D_{l(2),r(1),l(1)})  \\
            & \ \ - \delta T_{R_1}(D_{l(0),r(2)}) + \delta^2 T_{R_1}(D_{l(0),r(0)}) + \delta^2 T_{R_1}(D_{l(0),r(1)})  \\
            & \ \ - \delta T_{R_1}(D_{l(1),r(2)})+ \delta^2 T_{R_1}(D_{l(1),r(0)}) + \delta^2 T_{R_1}(D_{l(1),r(1)})  \\
            &= T_{R_1}(D_{l(2),r(2)}) + \delta T_{R_1}(D_{l(2),r(0),l(2)}) + \delta T_{R_1}(D_{l(2),r(1),l(2)})   \\
            & \ \ - \delta T_{R_1}(D_{l(0),r(2)}) - \delta T_{R_1}(D_{l(1),r(2)})
\end{align*}

The last equality is due to the fact that $s(D_{l(0),r(0)})=s(D_{l(2),r(0),l(1)})$, $s(D_{l(0),r(1)})=s(D_{l(2),r(1),l(1)})$, $s(D_{l(1),r(0)})=s(D_{l(2),r(0),l(0)})$, and $s(D_{l(1),r(1)})=s(D_{l(2),r(1),l(0)})$. In the same way, if $R_2$ starts by smoothing the crossing $r$, we obtain the following. 

\begin{align*}
    T_{R_2}(D) &=- T_{R_2}(D_{r(2)}) + \delta T_{R_2}(D_{r(0)}) + \delta T_{R_2}(D_{r(1)}) \\
            &=  T_{R_2}(D_{r(2),l(2)}) - \delta T_{R_2}(D_{r(2),l(0)}) - \delta T_{R_2}(D_{r(2),l(1)})   \\
            & \ \ +  \delta T_{R_2}(D_{r(0)}) + \delta T_{R_2}(D_{r(1)})
\end{align*}

We then notice that $s(D_{r(0)})=s(D_{l(2),r(0),l(2)})$, $s(D_{r(1)})=s(D_{l(2),r(1),l(2)})$ and $s(D_{r(2),l(j)})=s(D_{l(j),r(2)})$ for every $j \in \lbrace 0,1,2\rbrace$, and therefore $T_{R_1}(D)=T_{R_2}(D)$.

                    \begin{figure}[H]
                \centering
                \includegraphics[scale=0.7]{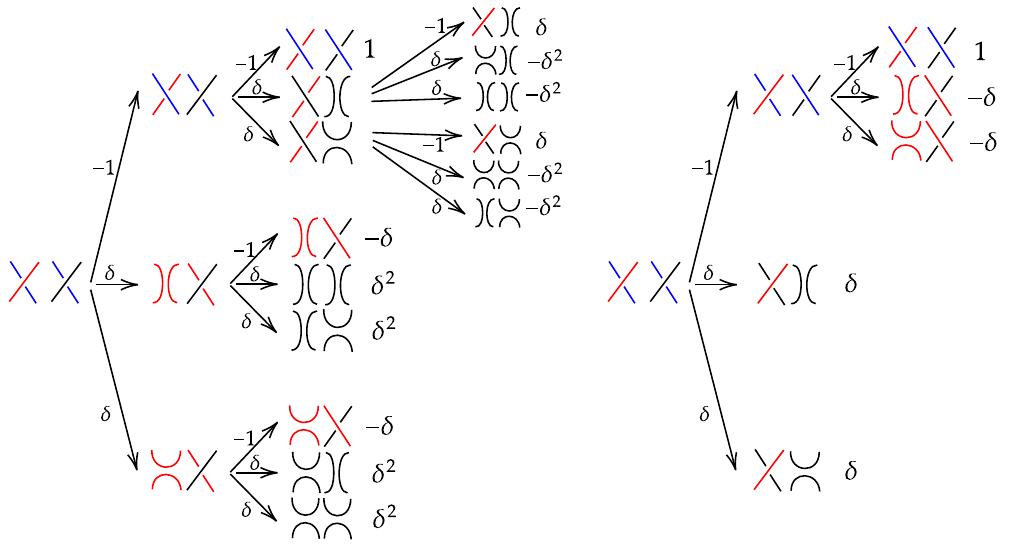}
                \caption{Resolution trees illustrating subcase (3c).} 
                \end{figure}

                    \item The proof of the case (3d) is analogous to that of case (3b).

                \end{enumerate}
            
        \end{enumerate}
\end{proof}

Thanks to the previous theorem, we know that for every tied link diagram $D$, we can associate a formal sum $$T_R(D)=\sum_{s\in S}f(s)\cdot s,$$ \noindent which does not depend on the tree $R$, so we will denote it as $T(D)$. 

\subsection{A regular isotopy invariant} We now use $T(D)$ to define a function $P$, and we show that it coincides with $\langle\langle \cdot \rangle\rangle$, thus proving that $P$ is a regular isotopy invariant.

\begin{definition}
    Let $D$ be a tied link diagram and $s$ an AJ-state of $D$. We define $$P(s)=c^{\gamma(s)-1}(-A^2-A^{-2})^{k(s)-\gamma(s)},$$ \noindent where $\gamma(s)$ denotes the number of colors of the AJ-state $s$, and $k(s)$ its number of components.
\end{definition}

For any tied link diagram $D$, we extend the definition of $P$ by setting $$P(D)=\sum_{s\in S} f(s)\,P(s),$$ \noindent where $S$ denotes the set of AJ-states of $D$.

\begin{prop}
    The polynomial $P$ coincides with $\langle\langle \cdot \rangle\rangle$, and therefore is a regular isotopy invariant.
\end{prop}

\begin{proof}
    We will show that $P$ satisfies axioms $(1)$ to $(6)$ of Theorem~\ref{invariante} except for (4) since, as mentioned in \cite{Aicardi2018}, (4) can be deduced from the others. Let $D$ be a tied link diagram, then we have that for any AJ-state $ s $ of $ D $, it holds that the smoothings that lead from $ D $ to $ s $ are the same as those that lead $ D \sqcup \bigcirc $ and $ D \widetilde{\sqcup} \bigcirc $ to $ s $, since their sets of crossings are exactly the same as $ D $. From this it follows that $ f(s \sqcup \bigcirc) = f(s \widetilde{\sqcup} \bigcirc) = f(s) $. Moreover, if $ S $ is the set of AJ-states of $ D $, then the AJ-states of $ D \sqcup \bigcirc $ and $ D \widetilde{\sqcup} \bigcirc $ are, respectively, $ \{ s \sqcup \bigcirc \mid s \in S \} $ and $ \{ s \widetilde{\sqcup} \bigcirc \mid s \in S \} $. Now, let us verify the axioms of $ \langle\langle \cdot \rangle\rangle $.

\begin{enumerate}
    \item \begin{align*}
        P(\bigcirc)=1.
    \end{align*}

    \item \begin{align*}
        P(D \sqcup \bigcirc) &=\sum_{s\in S}f(s\sqcup \bigcirc)P(s\sqcup \bigcirc)\\ 
        &=\sum_{s\in S}f(s\sqcup \bigcirc)P(s)c\\ 
        &=c\sum_{s\in S}f(s)P(s)\\ 
        &=c\cdot P(D).
    \end{align*}
    \item \begin{align*}
        P(D \widetilde{\sqcup} \bigcirc)&=\sum_{s\in S}f(s\widetilde{\sqcup} \bigcirc)P(s\widetilde{\sqcup} \bigcirc)\\ 
        &=\sum_{s\in S}f(s\widetilde{\sqcup} \bigcirc)P(s)(-A^2-A^{-2})\\ 
        &=(-A^2-A^{-2})\sum_{s\in S}f(s)P(s)\\ 
        &=(-A^2-A^{-2})\cdot P(D).
    \end{align*}
    \setcounter{enumi}{4}
    \item Given that $T(D_{+,\sim})=AT(D_0)+A^{-1}T(D_{\infty})$, we have
    \begin{align*}
        P(D_{+,\sim})=AP(D_0)+A^{-1}P(D_{\infty}).
    \end{align*}
    \item Given that $T(D_+)=-T(D_-)+(A+A^{-1})T(D_0)+(A+A^{-1})T_{\infty}$, we have
    \begin{align*}
        P(D_+)=-P(D_-)+(A+A^{-1})P(D_0)+(A+A^{-1})P(D_{\infty}).
    \end{align*}  
\end{enumerate}
\end{proof}

\subsection{Algorithm to compute $\langle\langle\cdot\rangle\rangle$}

We provide an algorithm to compute the double bracket $\langle\langle D\rangle\rangle$ of any tied link diagram $D$. This algorithm made possible the computations appearing in the next section. It can be accessed in \cite{Algoritmo}.

\

\begin{spacing}{0.02}
\noindent
\HRule
\end{spacing}

\

\noindent \textbf{Input:} a tied link diagram $D$ with color indexed by $1,2, \ldots, m$
\begin{itemize}
\item Order the crossings of $D$ and set $P=0$, $v=D$ and create an empty list to store processed vertices.
\item For every vertex $v$ which has not been processed:
 \begin{itemize}
    \item[$\cdot$] Run through the crossings of $v$ in order.
    \item[$\cdot$] If an illegal crossing of type 2 is found, smooth it following axiom $(6)$, set $v$ as processed, add the children to the list of improcessed vertices and, for every child $w$, store the parent $(w)=v$, and the label of the edge from $v$ to $w$.
    \item[$\cdot$] Otherwise, if an illegal crossing of type 2 is found, smooth it following axion $(5)$, and set $v$ as a leaf. Compute $P(s(v))$. Compute the product $b(v)$ of the labels of the edges connecting $D$ to $v$. Add $b(v)P(s(v))$ to $P$.
    \item[$\cdot$] Otherwise, $v$ is a leaf (i.e., an AJ-state). Compute $P(s(v))$. Compute the product $b(v)$ of the edge labels joining $D$ to $v$. Add $b(v)P(s(v))$ to $P$.
\end{itemize}
\end{itemize}

\noindent \textbf{Return:} P

\ 

\begin{spacing}{0.02}
\noindent
\HRule
\end{spacing}

\ 

\ 

While Theorem~\ref{independencia} assures us that any resolution tree can be used to calculate the value of the AJ-bracket, the interest in using the one described in the algorithm lies in its ability to produce a subtree with the same root, where the leaves contain only illegal crossings of type 1. This, in turn, makes it easier to distinguish between two non-tie-isotopic tied links.

\section{Some key examples and explicit calculation of $\langle\langle\cdot\rangle\rangle$}

In \cite[Section~7]{Aicardi2018}, Aicardi and Juyumaya provide three pairs of examples (a single one if orientations are not taken into account) showing the strength of the invariant $\mathcal{J}$, the \textit{tied Jones polynomial} obtained by normalizing the AJ-bracket. More precisely, they demonstrate that $\mathcal{J}$ distinguishes pairs of links that neither the Homflypt nor the two-variable Kauffman polynomials can distinguish. These links are represented by the diagrams shown in Figure \ref{contraejemplos} with three different (pairs of) orientations.

\begin{figure}[H]
    \centering
    \includegraphics[scale=0.75]{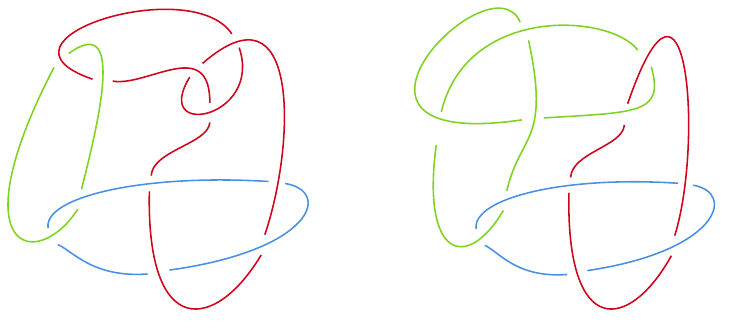}
    \caption{Diagrams of L11n304 and L11n412.}
    \label{contraejemplos}
\end{figure}

The algorithm described in the previous section that we implemented in \cite{Algoritmo} allowed us to find an inconsistency in their computations: the AJ-bracket of both diagrams coincide. More precisely, if we write $D_{L11n304}$ and $D_{L11n412}$ for the diagrams shown in Figure \ref{contraejemplos}:

\begin{align*}
    \langle\langle D_{L11n304}\rangle\rangle &= \langle\langle D_{L11n412} \rangle\rangle = A^{19} - 3A^{15} - A^{13}c + 2A^{13} + 2A^{11}c + 6A^{11} + 4A^9c \\
    & \ \ \ \ \ \ + A^7c^2 - 2A^7c - 4A^7 - A^5c - A^3c^2 + 2A^3c + 6A^3 - 2A - 4c/A \\
    & \ \ \ \ \ \ - 11/A - 5c/A^3 - 4/A^3 - c^2/A^5 + 2/A^5 - c/A^7 - 2/A^7 - 2c/A^9 \\
    & \ \ \ \ \ \ - 7/A^9 - c/A^{11} - 2/A^{11} + 3/A^{13} + c/A^{15} - 1/A^{17}
\end{align*}

In each of the pairs of tied diagrams in \cite{Aicardi2018} the writhes of both diagrams coincide, and therefore our previous computation implies that the polynomial invariant $\mathcal{J}$, does not distinguish between the associated links, that is:

\begin{equation*}
    \mathcal{J}(L11n304)=(-A)^{3w}\langle\langle D_{L11n304}\rangle\rangle=(-A)^{3w}\langle\langle D_{L11n412} \rangle\rangle=\mathcal{J}(L11n412)
\end{equation*}

So we wondered whether the invariant $\langle\langle\cdot\rangle\rangle$ is stronger than the Kauffman bracket. Subsequently, in discussions with one of the authors of the aforementioned work, we were advised to test examples found in \cite{Aicardi2020}, where authors studied the strengt of other polynomial invariants for tied links. 

The following table shows some pairs of non-equivalent classical links with three components sharing the same writhe (for both members of each pair). These examples cannot be distinguished by their Homflypt polynomial; however, the double bracket $\langle\langle\cdot\rangle\rangle$ distinguish them, and so does the tied Jones polynomial $\mathcal{J}$. For each link, we took the standard diagram appearing in \cite{linkinfo}.

\begin{table}[H]
\label{Tabla}
\centering
\begin{tblr}{
  width = \linewidth,
  colspec = {Q[200]Q[200]Q[781]},
  row{1} = {c},
  cell{2}{1} = {c},
  cell{2}{2} = {c},
  cell{2}{4} = {c},
  cell{3}{1} = {c},
  cell{3}{2} = {c},
  cell{3}{4} = {c},
  cell{4}{1} = {c},
  cell{4}{2} = {c},
  cell{4}{4} = {c},
  cell{5}{1} = {c},
  cell{5}{2} = {c},
  cell{5}{4} = {c},
  cell{6}{1} = {c},
  cell{6}{2} = {c},
  cell{6}{4} = {c},
  hlines,
  vlines,
}
$D_{1}$ & $D{2}$ & $\langle\langle D_1\rangle\rangle-\langle\langle D_2\rangle\rangle$  \\ 

L11n358\{0,1\} & L11n418\{0,0\} &  $A^{17} + A^{15} c + A^{15} + A^{13} c - A^{13} - 2 A^{11} c - A^{9} c + 2 A^{7} c - 2 A^{7} - A^{5} c - A^{5} - 3 A^{3} c - A^{3} + \frac{3 c}{A} + \frac{1}{A} + \frac{c}{A^{3}} + \frac{1}{A^{3}} - \frac{2 c}{A^{5}} + \frac{2}{A^{5}} + \frac{c}{A^{7}} + \frac{2 c}{A^{9}} - \frac{c}{A^{11}} + \frac{1}{A^{11}} - \frac{c}{A^{13}} - \frac{1}{A^{13}} - \frac{1}{A^{15}}$ \\ 

L11n358\{1,1\} & L11n418\{1,0\} &  $ A^{17} + A^{15} c + A^{15} + A^{13} c - A^{13} - 2 A^{11} c - A^{9} c + 2 A^{7} c - 2 A^{7} - A^{5} c - A^{5} - 3 A^{3} c - A^{3} + \frac{3 c}{A} + \frac{1}{A} + \frac{c}{A^{3}} + \frac{1}{A^{3}} - \frac{2 c}{A^{5}} + \frac{2}{A^{5}} + \frac{c}{A^{7}} + \frac{2 c}{A^{9}} - \frac{c}{A^{11}} + \frac{1}{A^{11}} - \frac{c}{A^{13}} - \frac{1}{A^{13}} - \frac{1}{A^{15}}$ \\ 

L11n356\{1,0\} & L11n434\{0,0\} &  $ - A^{13} - A^{11} c + 2 A^{9} + 3 A^{7} c - 3 A^{3} c - A + \frac{2 c}{A} - \frac{1}{A^{3}} - \frac{3 c}{A^{5}} + \frac{3 c}{A^{9}} + \frac{2}{A^{11}} - \frac{c}{A^{13}} - \frac{1}{A^{15}}$ \\ 

L11n325\{1,1\} & L11n424\{0,0\} & $ A^{19} + A^{17} c - A^{15} - 2 A^{13} c + 2 A^{13} + 2 A^{11} c + 2 A^{9} c - 2 A^{9} - 4 A^{7} c - A^{7} - 3 A^{5} c - 2 A^{5} + 2 A^{3} c + 3 A c - \frac{2 c}{A} + \frac{1}{A} - \frac{2 c}{A^{3}} + \frac{2}{A^{3}} + \frac{4 c}{A^{5}} + \frac{2 c}{A^{7}} + \frac{2}{A^{7}} - \frac{2 c}{A^{9}} + \frac{1}{A^{9}} - \frac{c}{A^{11}} - \frac{2}{A^{11}} - \frac{1}{A^{13}}$  \\ 

L10n79\{1,1\} & L10n95\{1,0\} &  $ - A^{18} - A^{16} c + A^{12} c - 2 A^{12} -2 A^{10} c - A^{8} c + 2 A^{6} c + A^{6} + 2 A^{4} c + 2 A^{4} + A^{2} - c + 2 + \frac{2 c}{A^{2}} + \frac{c}{A^{4}} - \frac{2 c}{A^{6}} - \frac{c}{A^{8}} - \frac{2}{A^{8}} - \frac{1}{A^{10}} $
\end{tblr}
\caption{}
\end{table}

\begin{figure}[H]
    \centering
    \includegraphics[scale=0.75]{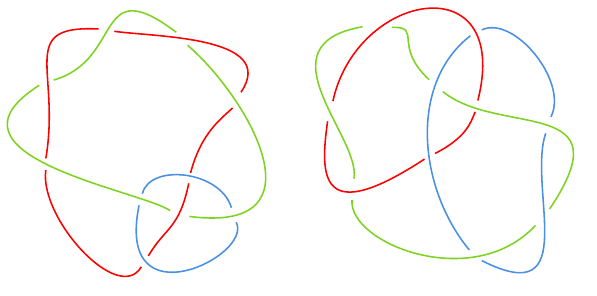}
    \caption{Diagrams of L10n79 and L10n95.}
\end{figure}

We have tried to find pairs of tied link diagrams with the same $2$-variable Kauffman polynomial and different tied Jones polynomial. To do so, we computed the value of $\langle\langle \cdot \rangle\rangle$ for all pairs of (oriented) link diagrams in \cite{linkinfo} with crossing number at most $11$ having the same Kauffman polynomial, considering the finest possible partition (that is, the number of colors equals the number of components). However, the value of $\mathcal{J}$ is the same for both members in each of those pairs. It is an open question to find pairs of links with the same $2$-variable Kauffman polynomial and different value of $\mathcal{J}$.

\section*{Acknowledgements}

The first author thanks J. González-Meneses and M. Silvero for helpful comments and corrections, as well as J. Juyumaya for providing examples for the final section of the article. O'Bryan Cárdenas-Andaur is supported by ANID, Beca Chile Doctorado en el Extranjero, Folio 72220167, and partially supported by the grant PID2020-117971GB-C21 funded by MCIN/AEI/10.13039/501100011033.

\bibliographystyle{amsplain}
\bibliography{main}{}

\end{document}